\documentclass{amsart}
\usepackage{fullpage}
\usepackage{psfrag}  
\usepackage{amsmath,amsthm}
\usepackage{mathtools}
\usepackage{marvosym}
\usepackage{graphicx}
\usepackage{xcolor}
\usepackage{epsfig}
\usepackage{enumerate}
\usepackage{cite}
\usepackage[active]{srcltx}

\usepackage{url}
\usepackage{hyperref}
\hypersetup{
	colorlinks,
	linkcolor={red!50!black},
	citecolor={blue!50!black},
	urlcolor={blue!80!black}
}  

\theoremstyle{plain}
\newtheorem{theorem}{Theorem}

\newtheorem{conj}[theorem]{Conjecture}

\newtheorem{lemma}[theorem]{Lemma}

\newtheorem{question}[theorem]{Question}

\theoremstyle{definition}

\newtheorem{remark}[theorem]{Remark}
\newtheorem{ack}[theorem]{Acknowledgments}

\newcommand{\comment}[1]{}

\newcommand{\bdry}{\ensuremath{\partial}}

\newcommand{\nbhd}{\ensuremath{\mathcal{N}}}

\newcommand{\Z}{\ensuremath{\mathbb{Z}}}

\definecolor{purple}{rgb}{0.4,0,0.6}

\title{Exceptional surgeries in $3$--manifolds}

\author{Kenneth L.\ Baker}
\address{Department of Mathematics\\University of Miami\\ Coral Gables, FL 33146 \\ USA}
\email{k.baker@math.miami.edu}

\author{Neil R.\ Hoffman}
\address{Department of Mathematics\\ Oklahoma State University\\ Stillwater, OK  74078 \\ USA}
\email{neil.r.hoffman@okstate.edu}


\begin{document}

\begin{abstract}
Myers shows that every compact, connected, orientable $3$--manifold with no $2$--sphere boundary components contains a hyperbolic knot. We use work of Ikeda with an observation of Adams-Reid to show that every $3$--manifold subject to the above conditions contains a hyperbolic knot which admits a non-trivial non-hyperbolic surgery, a toroidal surgery in particular.  We conclude with a question and a conjecture about reducible surgeries.	
\end{abstract}

\maketitle


Myers shows that there are hyperbolic knots in every compact, connected, orientable $3$--manifold whose boundary contains no $2$--spheres \cite{MYERSEXCELLENT}.
Might there be such a $3$--manifold for which every hyperbolic knot has no non-trivial exceptional surgeries?  

One approach to showing the answer is {\bf Yes} would be to prove that there exists a $3$--manifold in which every hyperbolic knot has cusp volume larger than $18$ so that the $6$--Theorem \cite{agol2000, lackenby2000word} would obstruct any exceptional surgery.  However, \cite[Corollary 5.2]{ACFGK-cuspsizebounds} implies that every closed, connected, orientable $3$--manifold contains infinitely many hyperbolic knots with cusp volume at most $9$.  So this approach will not work.   Furthermore, the knots constructed in 
\cite[Corollary 5.2]{ACFGK-cuspsizebounds} do not necessarily have any exceptional surgery, so that work does not address our question.  

In this short note we demonstrate the answer to the question is actually {\bf No} by constructing hyperbolic knots with a non-trivial toroidal surgery in any $3$--manifold.

\begin{theorem}\label{thm:main}
	Let $M$ be a compact, connected, orientable $3$--manifold such that $\bdry M$ contains no $2$--spheres. There exist infinitely many hyperbolic knots in $M$ that admit a toroidal surgery.   	
\end{theorem}

\begin{proof}
	Let $M$ be a compact, connected orientable $3$--manifold whose boundary contains no $2$--spheres.  In \cite{ikeda2012hyperbolic}, Ikeda shows that $M$ contains an infinite family of embedded genus $2$ handlebodies in $M$, each with hyperbolic and anannular complement of its interior where its genus $2$ boundary is totally geodesic.
	Let $H$ be any one of these handlebodies.

	In Lemma~\ref{lem:hypknot} we find a knot $K$ in $H$ that bounds an embedded once-punctured Klein bottle $\Sigma$ such that $H-K$ is a one-cusped anannular hyperbolic manifold in which $\bdry H$ is totally geodesic.   Therefore $M-K$ decomposes along $\bdry H$ into two anannular hyperbolic manifolds.   Thus, following an observation of Adams and Reid \cite[Observation 2.1]{adams_reid_1993},  $M-K$ is a hyperbolic manifold containing a quasi-Fuchsian surface isotopic to $\bdry H$, and $K$ is a hyperbolic knot in $M$.  (Note that while $\bdry H$ is totally geodesic in both $M-int(H)$ and $H-K$, its hyperbolic structure may not be the same in these two manifolds.  Hence $\bdry H$ is not necessarily totally geodesic in $M-K$.) 
	
	Since $K$ bounds the once-punctured Klein bottle $\Sigma$, surgery on $K$ along the slope $\sigma$ of $\bdry \Sigma$ produces a manifold $M_K(\sigma)$ containing an embedded Klein bottle $\widehat{\Sigma}$.  The manifold $M_K(\sigma)$ will be toroidal unless the torus $\bdry \nbhd(\widehat{\Sigma})$ compresses.  However, Lemma~\ref{lem:hypknot} shows that $K$ may be further chosen in $M$ so that $H_K(\sigma)-\widehat{\Sigma}$ is also a one-cusped anannular hyperbolic manifold in which $\bdry H$ is totally geodesic boundary. Therefore, as  $M_K(\sigma) - \widehat{\Sigma}$ decomposes along $\bdry H$ into the hyperbolic manifolds $M-int(H)$ and $H_K(\sigma)-\widehat{\Sigma}$, it follows that $\bdry \nbhd(\widehat{\Sigma})$ must be incompressible in $M_K(\sigma)$.   
\end{proof}

\begin{lemma}
	\label{lem:hypknot}
	There is a knot $K$ in a genus $2$ handlebody $H$ that bounds a once-punctured Klein bottle $\Sigma$ so that $H-K$ is a one-cusped anannular hyperbolic manifold in which $\bdry H$ is totally geodesic.
	Hence surgery on $K$ along the slope $\sigma$ of $\bdry \Sigma$ produces a manifold $H_K(\sigma)$ containing an embedded Klein bottle $\widehat{\Sigma}$.
	Furthermore, $K$ may be chosen so that $H_K(\sigma) - \widehat{\Sigma}$ is also a one-cusped anannular hyperbolic manifold in which $\bdry H$ is totally geodesic.
\end{lemma}

\begin{proof}
	Figure~\ref{fig:knotext}(a) shows a surgery description of a trivial $3$--strand tangle in the ball, along with an arc $k$ that has its endpoints on the tangle strands. 
	Figure~\ref{fig:knotext}(b)	shows the result of an isotopy in which the tangle is more obviously trivial at the expense of elongating the arc $k$.  The double branched cover of this trivial  $3$--strand tangle is a handlebody $H$ in which the arc $k$ lifts to a knot $K$. Figure~\ref{fig:knotext}(c),(d), and (e) illustrate the construction of the knot $K$ in the handlebody $H$.  In (c), two caps with red dual arcs are attached to the $3$--strand tangle to form a trivial $1$--strand tangle in the ball.  After straightening the strand in (d), the double branched cover is taken in (e).  The two caps each lift to $2$--handles attached to $H$.  The two red arcs in (e) are the co-cores of these two $2$--handles, so $H$ is obtained by drilling them out.  For lifting the surgery description, note that a curve linking the branch locus once with surgery coefficient $1/2a$ lifts to a single curve with surgery coefficient $1/a$.  Thus for each pair of integers $n,m$, we obtain a knot $K$ in a genus $2$ handlebody $H$.
	
	Figure~\ref{fig:knotext}(f) shows a surgery description of the double of $(H,K)$ across $\bdry H$, the link $K \cup \overline{K}$ in $H \cup \overline{H} = S^1 \times S^2 \# S^1 \times S^2$, obtained by mirroring Figure~\ref{fig:knotext}(e) and performing $0$ surgery on the components formed from the co-cores of the $2$--handles and their mirrors.
	The certificate for a verified computation in SnapPy \cite{SnapPy} confirms that the complement of the link $K \cup \overline{K}$ in $S^1 \times S^2 \# S^1 \times S^2$ is hyperbolic for choices of $n,m \in \Z$ with $|n|$ and $|m|$ suitably large. We provide that certificate as an ancillary file to this arxiv posting. 
	(More specifically, a verified computation in SnapPy shows that after doing the two $0$-surgeries on the two red components in Figure~\ref{fig:knotext}(f), the resulting  $6$--component link in $S^1 \times S^2 \# S^1 \times S^2$ has a hyperbolic complement. Then there is a constant $N$ such that the $2$--component link complement resulting from the surgeries on the green and purple components will be hyperbolic if both $|n|>N$ and $|m|>N$; see  \cite[Lemma
	5]{kojima1988isometry} or \cite[Theorem 3.1]{BHL2019jointly}.)
	Since the double has the reflective symmetry in which $\bdry H$ is the fixed set, it must be a totally geodesic surface.  Hence $H-K$ must be a one-cusped anannular hyperbolic manifold in which $\bdry H$ is totally geodesic.
		
	In Figure~\ref{fig:knotext}(e) one observes that $K$ bounds a once-punctured Klein bottle $\Sigma$ in $H$ that is disjoint from the two curves of the surgery description. As such, Dehn surgery on $K$ in $H$ along the boundary slope $\sigma = \bdry \Sigma$ produces the manifold $H_K(\sigma)$ which contains the Klein bottle $\widehat{\Sigma}$ obtained by capping off $\Sigma$ with a meridional disk of the surgery.

	All that remains is to show that $\widehat{\Sigma}$ is essential in the filling. First, we may understand the complement of $\widehat{\Sigma}$ through tangles.
	As apparent in Figure~\ref{fig:knotext}(e), the surface $\Sigma$ may be taken to be invariant under the involution of $H$ from the branched covering so that the fixed set intersects $\Sigma$ in two points and an arc.  Then $\Sigma$ descends to a disk $D$ containing the arc $k$ in its boundary and meeting the branch locus in the remainder of its boundary and two points in its interior.  This disk $D$ may be tracked from its initial quotient of $\Sigma$ in Figure~\ref{fig:knotext}(d) back to  Figure~\ref{fig:knotext}(a).  
	Now Figure~\ref{fig:tanglecomplements}(a) shows the exterior of the arc $k$ while Figure~\ref{fig:tanglecomplements}(b) shows the rational tangle filling associated to $\sigma$--framed surgery on $K$.  In particular, the disk $D-k$ is completed to a disk $\widehat{D}$ containing the closed component of the branch locus as its boundary and meeting the strands of the branch locus in two interior points.   Indeed, the double branched cover of the tangle Figure~\ref{fig:tanglecomplements}(b) is the manifold $H_K(\sigma)$ in which $\widehat{D}$ lifts to $\widehat{\Sigma}$.
	Finally, Figure~\ref{fig:tanglecomplements}(c) shows the tangle that is the complement of a small regular neighborhood of $\widehat{D}$.

	Figure~\ref{fig:kleinext}(a) shows a rational tangle filling of Figure~\ref{fig:tanglecomplements}(c) with the arc $k'$ that is the core of the rational tangle.  This $3$--strand tangle is a trivial tangle as made more apparent in Figures~\ref{fig:kleinext}(b), (c), and (d) which isotop the tangle while elongating arc $k'$. As before, (c) shows the attachment of two caps with dual arcs and (d) straightens the resulting $1$--strand tangle. Figure~\ref{fig:kleinext}(e) shows the double branched cover which illustrates the lift of the arc $k'$ as the knot $K'$ in another genus $2$ handlebody $H'$.  Again, the two caps each lift to $2$--handles attached to $H'$, the two red arcs in (e) are the co-cores of these two $2$--handles, and so $H'$ is obtained by drilling them out.  Note that the knot $K'$ in $H'$ depends on the previously chosen pair of integers $n,m$ of the surgery description.

	It now follows that, by construction, $H_K(\sigma) - \widehat{\Sigma}$ is homeomorphic to $H'-K'$.   We show that $H'-K'$ is a one-cusped anannular hyperbolic manifold in which $\bdry H'$ is totally geodesic just as we did for $H-K$.
	Figure~\ref{fig:kleinext}(f)  shows a surgery description of the double of $(H',K')$ across $\bdry H'$, the link $K' \cup \overline{K'}$ in $H' \cup \overline{H'} = S^1 \times S^2 \# S^1 \times S^2$, obtained by mirroring Figure~\ref{fig:kleinext}(e) and performing $0$ surgery on the components formed from the co-cores of the $2$--handles and their mirrors.
	Just as before, the certificate for a verified computation in SnapPy \cite{SnapPy} confirms that the complement of the link $K' \cup \overline{K'}$ in $S^1 \times S^2 \# S^1 \times S^2$ is hyperbolic if both $|n|$ and $|m|$ are suitably large. We provide that certificate as an ancillary file to this arxiv posting. Since the double has the reflective symmetry in which $\bdry H'$ is the fixed set, it must be a totally geodesic surface.  Hence $H'-K'$ is a one-cusped anannular hyperbolic manifold in which $\bdry H'$ is totally geodesic.

	Since $H_K(\sigma) - \widehat{\Sigma} \cong H'-K'$, we obtain the desired results whenever $|n|$ and $|m|$ are large enough to be suitably large in both situations.  
	\end{proof}

\begin{remark}
To give a concrete example, taking $n=m=1$ is sufficient for the knots $K \subset H$ and $K' \subset H'$ to be hyperbolic.  Certainly, one could verify the hyperbolicity of these knots by hand in the spirit of what was done in \cite{adams_reid_1993}, but the argument would take longer.  Hence we content ourselves with verified computations in SnapPy \cite{SnapPy}.
\end{remark}

\begin{figure}
	\centering
	\includegraphics[width=\textwidth]{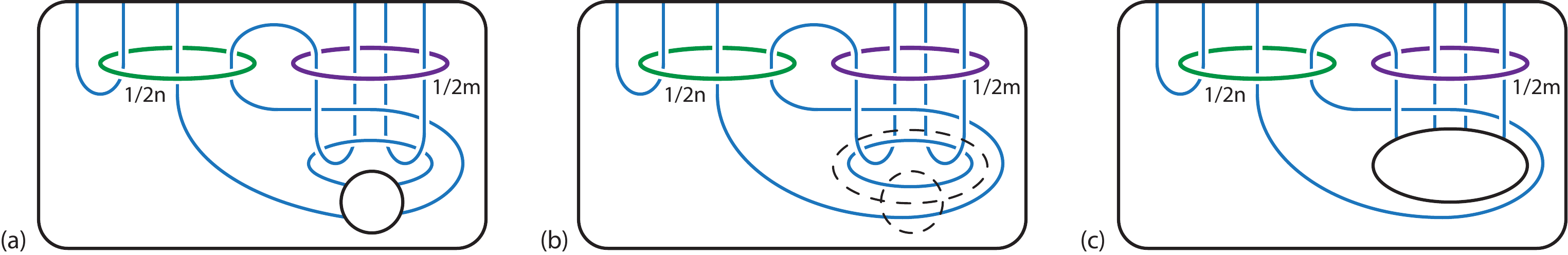}
	\caption{(b) A surgery description of a $3$--strand tangle in the ball with an unknot component that bounds a disk intersected twice by the strands.  (a) The complement of a rational tangle in this tangle.  (c) The complement of a small neighborhood of the disk bounded by the unknot component.}
	\label{fig:tanglecomplements}
\end{figure}

\begin{figure}
	\centering
	\includegraphics[width=\textwidth]{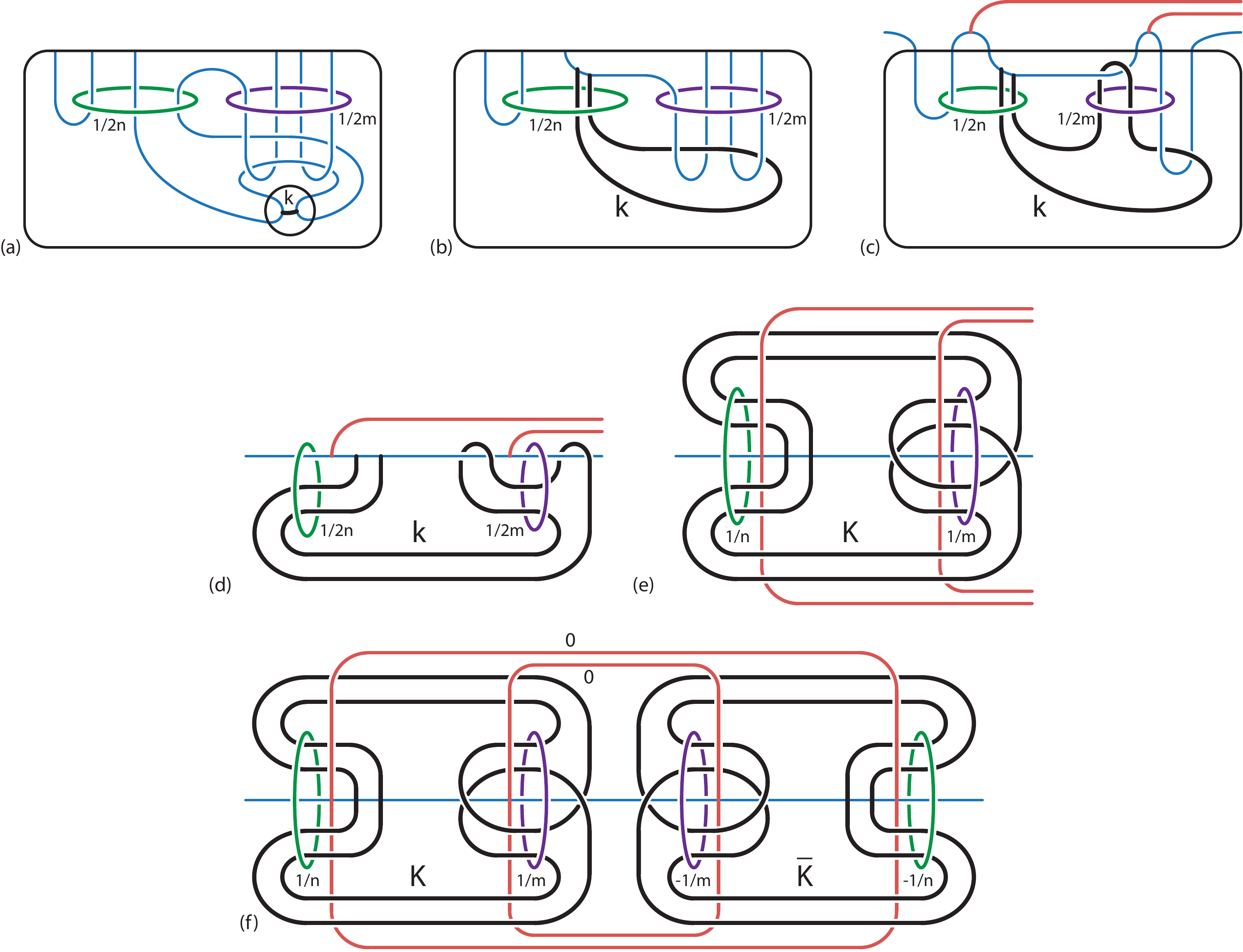}
	\caption{(a) A rational tangle filling of Figure~\ref{fig:tanglecomplements}(a) with its core arc $k$.  (b) \& (c) An isotopy showing the filled tangle is a rational $3$--strand tangle.  The arc $k$ is carried along.  (c) Attached to the rational $3$--strand tangle are two caps with their dual arcs to form a $1$--strand tangle in the ball.  (d)  The $1$--strand tangle is straightened.  (e) The double branched cover is formed. Drilling out the red arcs leaves a genus $2$ handlebody $H$ containing the knot $K$ that covers $k$. (f) A surgery description of the double of $(H,K)$ is formed. }
	\label{fig:knotext}
\end{figure}

\begin{figure}
	\centering
	\includegraphics[width=\textwidth]{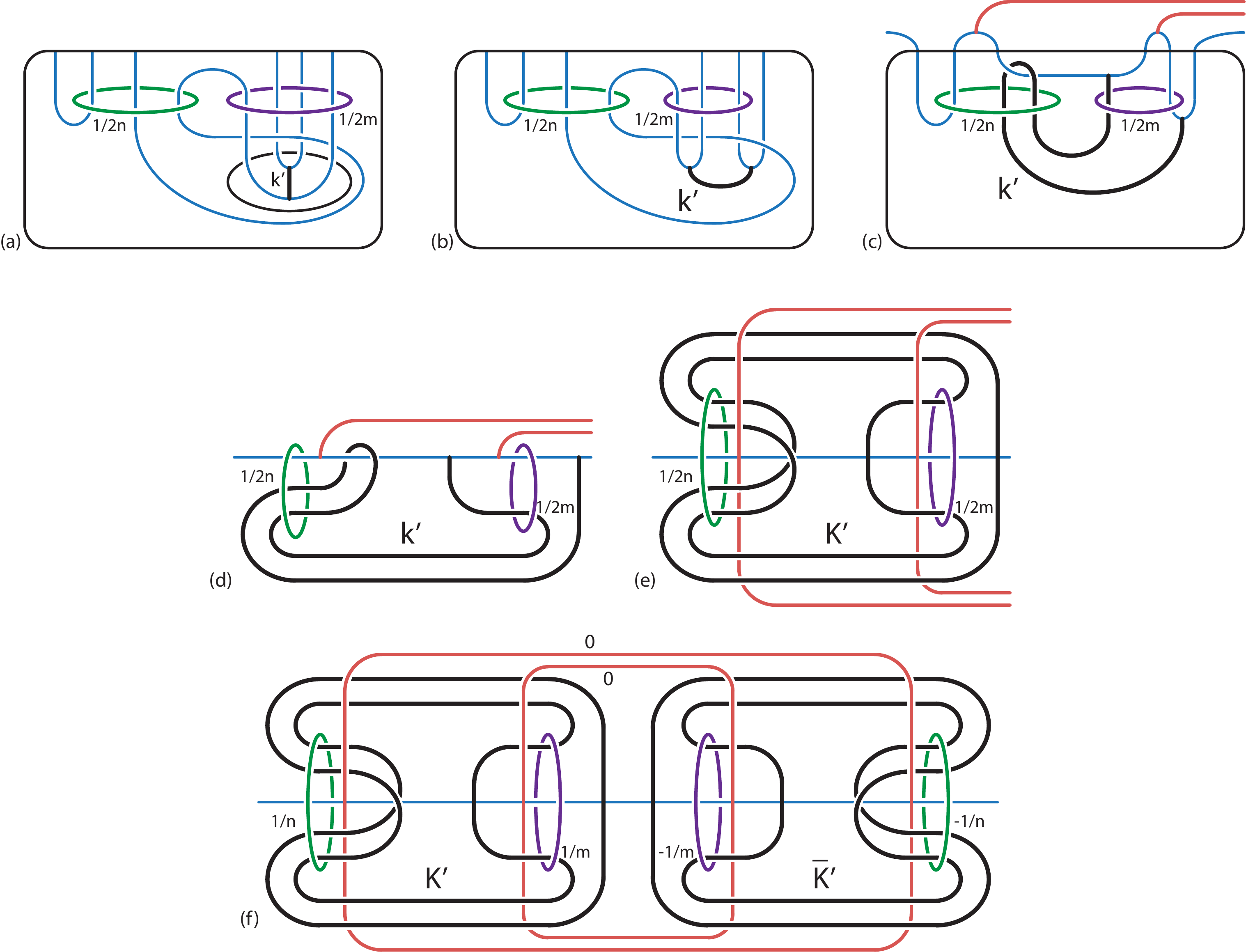}
	\caption{(a) A rational tangle filling of Figure~\ref{fig:tanglecomplements}(c) with its core arc $k'$.  (b) \& (c) An isotopy showing the filled tangle is a rational $3$--strand tangle.  The arc $k'$ is carried along.  (c) Attached to the rational $3$--strand tangle are two caps with their dual arcs to form a $1$--strand tangle in the ball.  (d)  The $1$--strand tangle is straightened.  (e) The double branched cover is formed. Drilling out the red arcs leaves a genus $2$ handlebody $H'$ containing the knot $K'$ that covers $k'$. (f) A surgery description of the double of $(H',K')$ is formed. }
	\label{fig:kleinext}
\end{figure}

What can be said about other kinds of exceptional surgeries?  Considerations of Betti numbers show that many closed, compact, orientable $3$--manifolds cannot contain a knot with a Dehn surgery to a lens space or a small Seifert fibered space.
In light of the Cabling Conjecture \cite{GAS} whose proof would imply that no hyperbolic knot in $S^3$ has a reducible surgery, it is reasonable to expect that there are $3$--manifolds in which no hyperbolic knot admits a reducible surgery. However, we are presently unaware of any $3$--manifold known to not have a hyperbolic knot with a non-trivial reducible surgery.

\begin{question}
	Which compact, connected, orientable $3$--manifolds do not contain a hyperbolic knot with a non-trivial reducible surgery?
\end{question}

While non-trivial reducible surgeries on hyperbolic knots in reducible manifolds do exist, see e.g.\ \cite{hoffmanmatignon}, we suspect that manifolds whose prime decompositions have at least $3$ summands are candidates.  

\begin{conj}
	A closed orientable $3$--manifold with at least $3$ summands does not contain a hyperbolic knot with a non-trivial reducible surgery.
\end{conj}

Towards the conjecture, suppose $K$ is a hyperbolic knot in a closed orientable $3$--manifold $M$ with at least $3$ summands.  
One may hope that each planar meridional surfaces in the  knot complement $M-K$ arising from $K$ intersecting multiple reducing spheres would contribute a certain amount to the length of the shortest longitude of $K$.  From this, at least if $M$ had sufficiently many summands, one would be able to use the 6-Theorem to obstruct a non-trivial reducible surgery. However this would also obstruct a toroidal surgery contrary to Theorem~\ref{thm:main}.  Indeed, it would also contradict  \cite[Corollary~5.2]{ACFGK-cuspsizebounds} which shows that the topology of  $M$ cannot force all longitudes of hyperbolic knots in $M$ to be long.   

On the other hand, combinatorial structures in knot complements can induce obstructions.   For instance, 
hyperbolic alternating knots in $S^3$ that have at least $9$ twist regions (in twist-reduced diagrams) provide an obstruction the existence of non-trivial exceptional fillings; see \cite[Theorem 5.1]{lackenby2000word}.


\begin{ack}
KB thanks Jacob Caudell for conversations related to \cite[Conjecture 5]{caudell} that prompted this note.
This work was partially supported by Simons Foundation grant \#209184 to Kenneth Baker and by Simons Foundation grant \#524123 to Neil Hoffman.
\end{ack}

\bibliographystyle{alpha}
\bibliography{toroidalsurgeries}

\end{document}